\documentclass[12pt]{amsart}

\newenvironment{rk}{\vspace{2 ex}{\noindent{\bf Remark}}}
        {\vspace{2 ex}}
\newtheorem{lemma}{Lemma}
\newtheorem{theorem}{Theorem}

\newtheorem{DEF}{Definition}

\usepackage{fullpage,graphicx}
\usepackage{amsfonts, mathrsfs}
\usepackage{latexsym, mathtools, esint, xcolor, soul}
\usepackage{amsmath}
\usepackage{amssymb}
\usepackage{subcaption}
\usepackage{epstopdf}

\newcommand{\nab}{\nabla}

\newcommand{\R}{\mathbb{R}}

\newcommand{\W}{\Omega}
\newcommand{\p}{\partial}
\newcommand{\s}{\sigma}
\newcommand{\la}{\lambda}

\newcommand{\G}{\Gamma}
\newcommand{\D}{\Delta}
\newcommand{\id}{\mathrm{d}}

\begin{document}
\openup0.4\jot

\pagestyle{plain}

\title{Solutions to a two-dimensional, Neumann free boundary problem}

\author{J. A. Gemmer} \address{Wake Forest University}
\author{G. Moon} \address{University of North Carolina at Chapel Hill}
\author{S. G. Raynor} \address{Wake Forest University} \thanks{The third author would like to thank the Simons Foundation for their support during the creation of this work.} \email{raynorsg@wfu.edu}
\begin{abstract}
We explore regularity properties of solutions to a two-phase elliptic free boundary problem near a Neumann fixed boundary in two dimensions. Consider a function u, which is harmonic where it is not zero and satisfies a gradient jump condition weakly along the free boundary.  Our main result is that u is Lipschitz continuous up to the Neumann fixed boundary. We also present a numerical exploration of the way in which the free and fixed boundaries interact.
\end{abstract}

\maketitle

%%%%%%%%%%%%%%%%%%%%%%%%%%%%%%%%%%%%%%%%%%%%%%%%%%%%%%%%%%%%%%%
%	Introduction
%
%%%%%%%%%%%%%%%%%%%%%%%%%%%%%%%%%%%%%%%%%%%%%%%%%%%%%%%%%%%%%%

\section{Introduction}\label{intro}
In this paper we study the regularity of a two-phase free boundary with Neumann boundary conditions. A prototypical example of such a problem is the determination of steady state velocity fields for the laminar flow of two immiscible, incompressible fluids \cite{F}. It is a classical result that for each fluid there exists a corresponding velocity potential that satisfies Laplace's equation \cite{landau2013fluid}. However, to satisfy local stress balance, a gradient jump condition in the potential must be satisfied at the fluid-fluid interface \cite{landau2013fluid}. In Figure \ref{introfigure}(A) we plot on a square domain a generic example of velocity fields satisfying such properties. This problem also arises in a number of other applied areas including, but not limited to, fluid dynamics, electromagnetics and optimal shape design; see \cite{F,ACF2, flucher1997bernoulli,chen2015free} and the references therein.

\begin{figure}
\includegraphics[width=.8\textwidth]{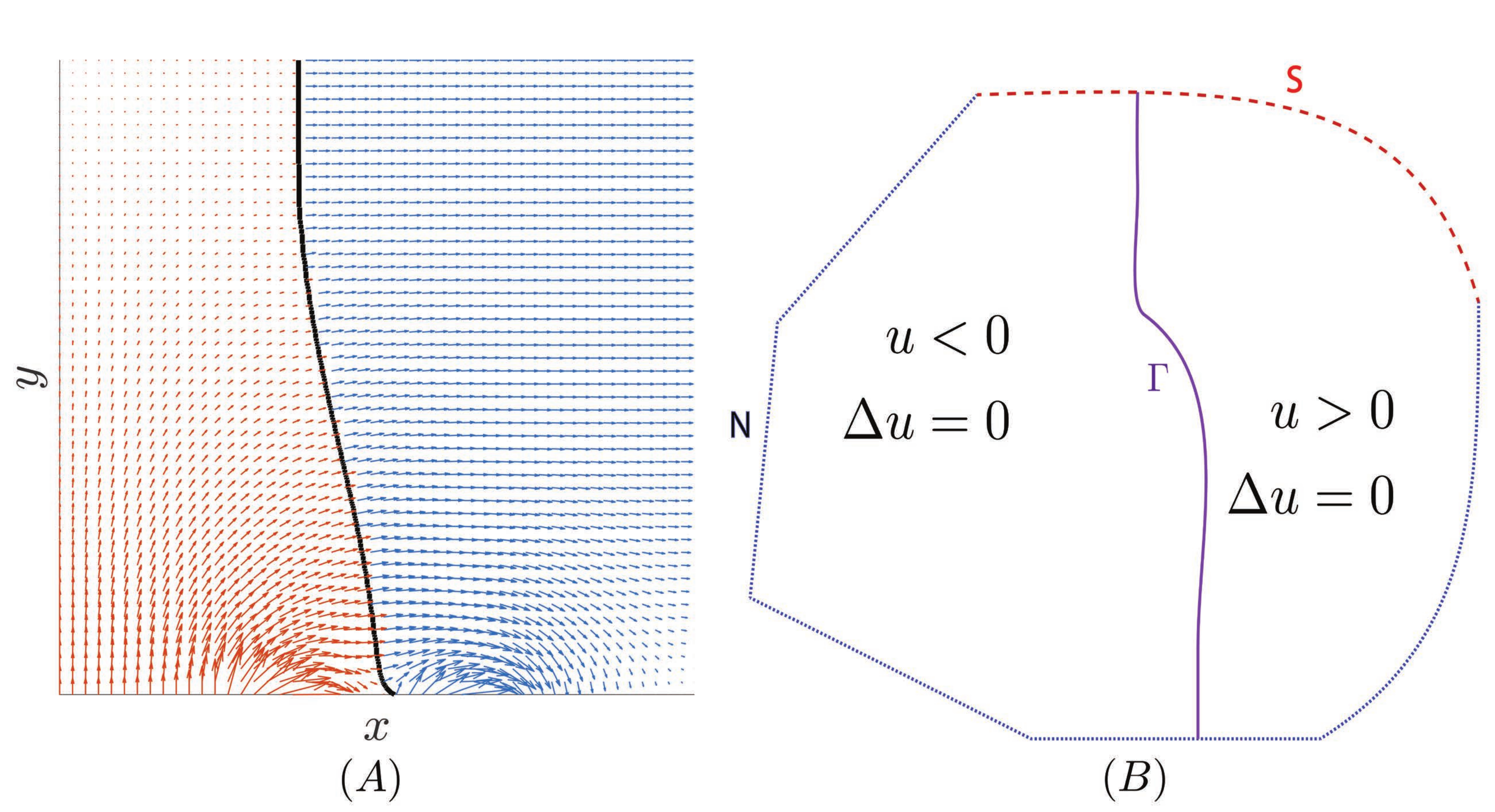}
\caption{(A) Velocity fields for two immiscible ideal fluids meeting at interface. Dirichlet boundary conditions were imposed on the bottom and right sides of the domain. Neumann boundary conditions were imposed on the left and top sides. (B) Schematic diagram of the free boundary problem. On $S$ and $N$ Dirichlet and Neumann Neumann boundary conditions are imposed respectively. The curve $\Gamma$ denotes the free boundary separating the phases $\{u>0\}$ and $\{u<0\}$.}
\label{introfigure}
\end{figure}

In words the problem is the following: find a function which is harmonic where it is nonzero and satisfies a possibly spatially inhomogeneous gradient jump condition across the boundary of its zero set $\Gamma$. Because the condition on $\Gamma$ is overdetermined, it is necessary not to predetermine the location of the transition--hence the name ``free boundary" problem and $\Gamma$ is known as the free boundary; see Figure \ref{introfigure}(B) for a schematic diagram of this problem. We are specifically interested in how the free boundary interacts with Neumann corner points on the boundary. For a smooth harmonic function satisfying Neumann boundary conditions it is clear that the level sets of the function, and in particular the free boundary, will intersect orthogonally with the Neumann boundary. However, depending on the opening angle of the corner, this local condition of orthogonality may contradict continuity of the free boundary away from the corner point. 

We will study this problem from a variational perspective. Namely, we will consider local minimizers of the functional $J:K\mapsto \mathbb{R}^+$ defined by
\begin{equation}\label{Eqn:EnergyFunc}
J[v]:=\int_{\W} \left ( |\nab v|^2 + Q^2(x) \lambda^2(v) \right ) d x,
\end{equation}
where for $\lambda_1>\lambda_2>0$ the function $\lambda :\mathbb{R}\mapsto \mathbb{R}^+$ is defined by
\begin{equation}\label{Eqn:Lambda}
\lambda(v)=\begin{cases} \lambda_1 & v > 0 \\ \lambda_2 & v \leq 0 \\ \end{cases},
\end{equation}
$\Omega\subset \mathbb{R}^2$ is a bounded, convex domain 
and $Q:\overline{\Omega}\mapsto \mathbb{R}^+$ is a measurable function satisfying for almost all $x\in \overline{\Omega}$:
\begin{equation}
0<m\leq Q(x) \leq M < \infty.
\end{equation}
The admissible set $K$ is defined by
\begin{equation}
K=\left\{v\in H^1(\Omega): \left. v\right|_{S}=u_0 \right\},
\end{equation}
where $u_0\in H^1(\overline{\Omega})$ and $S \subsetneq \p \W$. The existence of minimizers for this problem can be established using the direct method of the calculus of variations \cite{ACF} and to prevent triviality of minimizers we assume that the Dirichlet boundary data $u_0$ is inhomogeneous on $S$.

The functional $J$ models the energy or velocity potential for a large class of two phase problems. Specifically, if we let $u$ denote a minimizer of $J$ and use $P_+=\{u>0\}$ and $P_-=\{u<0\}$ to denote sets of positive and negative phases respectively, then $u$ enjoys the following properties \cite{ACF}:
\begin{enumerate}
\item $\D u = 0$ on $P_+$ and $P_-$,
\item $u=u_0$ on $S$,
\item $\displaystyle{\frac {\partial u}{\partial \nu}=0}$ in a weak sense along $N = \p\W \setminus S$,
\item On $\Gamma=\partial\{u>0\}$ the following jump condition is satisfied:
\begin{equation}\label{Eqn:GradJumpCondition}
|\nabla u^+|^2-|\nabla u^-|^2=(\lambda_1^2-\lambda_2^2)Q^2(x).
\end{equation}
\end{enumerate}
Formally, these properties arise as necessary conditions satisfied by critical points of the first variation of $J$. That is, the Neumann boundary conditions arise as the ``natural boundary conditions'' on $N$.  The gradient jump condition results from the fact that the distributional derivative of $\lambda^2(v)$ is a delta function of mass $\lambda_1^2-\lambda_2^2$.

The primary results we present in this paper are twofold. First, we prove that minimizers are Lipschitz continuous, a property that has also been shown to hold for Dirichlet boundary conditions \cite{ACF} and the one-phase Neumann problem \cite{R}. We restrict to $\R^2$ for technical reasons regarding the up-to-the-boundary monotonicity formula which we use to prove Lipschitz continuity.  The assumption that $\Omega$ is convex (but not necessarily smooth) is critical.  Indeed, even harmonic functions are not Lipschitz up to the boundary in non-convex, non-smooth domains. To see this, consider a harmonic function in a conic sector of $\R^2$ with opening angle $\theta$.  This function is proportional to $\|\vec{x}\|^\frac{\pi}{\theta},$ and when $\theta > \pi$ the resulting exponent is less than one, yielding a harmonic function that is not Lipschitz up to the vertex of the sector.  Therefore, to consider non-smooth domains we must impose the convexity condition. However, this convexity condition is truly necessary only near non-smooth points of the domain, so an exterior ball condition should be sufficient.   

Second, we numerically explore the interaction of the free boundary across Neumann boundaries containing corners, specifically parallelograms of various angles. By varying our Dirichlet boundary conditions on $S$ in such a way as to push the free boundary across a corner, we found that the free boundary does approach the fixed boundary orthogonally and will always do so.  However, as a perturbation in the fixed boundary conditions forces the free boundary to cross an acute angle, there is a jump in the position of the free boundary.  There is a forbidden region where the lack of room for an orthogonal intersection prevents the free boundary from intersecting the corner point. 

The numerical scheme we used is a simple finite difference approximation to the gradient flow applied to a relaxed version of $J$.  Here, $\lambda$ is replaced by a smooth transition layer.  This is a  technique used in \cite{caffarelli1995free} to model the temporal evolution of a propagating flame front. The benefits of using this approach are twofold. First, in contrast with shape optimization techniques \cite{haslinger2003shape, ito2008variational} and level set methods \cite{bouchon2005numerical, kuster2007fast}, this scheme is easy to implement for our specific problem. Second, in contrast with classical front tracking techniques \cite{flucher1997bernoulli, galvis2015iterative}, topological changes in the free boundary are handled by default, since the free boundary is simply the zero contour of a function. The price we pay for ease of implementation is in computational time. Namely, since gradients along the free boundary vary rapidly in space, a fine spatial discretization $\Delta x$ is required. However, it follows from the Courant-Friedrichs-Lewy (CFL) condition that the time discretization $\Delta t$ must satisfy $\Delta t < C\Delta x^2$ \cite{strikwerda2004finite}.

This paper is organized as follows: In Section 2 we review certain basic properties of minimizers for this problem.  In Section 3 we prove the main theorem on Lipschitz continuity of the minimizer.  Section 4 contains the explanation of the numerical scheme with a proof of convergence and Section 5 contains the numerical results and a discussion of them. Section 6 provides a conclusion and discussion of possible future directions that arise from our work. 
 
%%%%%%%%%%%%%%%%%%%%%%%%%%%%%%%%%%%%%%%%%%%%%%%%%%%%%%%%%%%%%%%
%	Preliminaries
%
%%%%%%%%%%%%%%%%%%%%%%%%%%%%%%%%%%%%%%%%%%%%%%%%%%%%%%%%%%%%%%

\section{Preliminaries}\label{prelimsection}
%Throughout this paper, we will use the following notation. Let $a > 0$.  Then for a set $\W$, $\W_a$ will be the open set $\{ x \in \W : d (x, S) > a\}$.  Let $D \subset \R^n$ be a domain.  For a function $f \in H^1(D)$, and a measurable set $T\subset \p D$ of positive $(n-1)$-dimensional Hausdorff measure, $f|_T$ will denote the trace of the function $f$ along $T$, which is in $L^2(T)$.  The support of an $H^1$ function $f$ will be denoted by $\mathrm{supp}(f)$ .  Finally, $|D|$ denotes the $n$-dimensional Lebesgue measure of $D$ and 
%$$\fint_{D} f \ \id x \coloneqq \frac1{|D|} \int_{D}f \ \id x$$
% is the average of $f$ over $D$. 
 
%Let $\W$ be a bounded, connected subset of $\R^2$ whose boundary is locally Lipschitz. Let $S$ be a closed, proper, nonempty subset of $\p \W$. Let $u_0$ be a smooth function on $\R^{2}$, with $A_+ = \max\{\sup_S u_0,0\}$ and\\ $A_-=\min\{\inf_S u_0,0\}$.

%Define $\lambda_1 > \lambda_2 >0$ and $\La = \lambda_1^2-\lambda_2^2>0.$ 
% Set $$\lambda(s)=\begin{cases} \lambda_1 & s > 0 \\ \lambda_2  & s \leq 0\end{cases}.$$ Also, let $Q(x) \in L^\infty(\overline{\W})$ and $m,M > 0$ so that $m \leq Q(x) \leq M$ for almost every $x \in \W$.  Then we want to minimize the functional
% \begin{equation}\label{functional}
% J[v] = \int_{\W} \left ( |\nab v|^2 + Q^2(x) \lambda^2(v) \right ) dx,
% \end{equation}
% over the set $K = \{ v \in H^1(\W) : v|_{S} = u_0 \}$. 
The free boundary of interest is
$$\G = \p\{ u > 0 \} \cap \W.$$
Considering a neighborhood $U$ of $\p\{ u < 0 \} \setminus \p\{ u > 0 \}$ we clearly must have $u \leq 0$ on $U$. However, $\la(0) = \la_{2}$ implies that $u$ must also be harmonic in $U$. Thus, we have $\p\{ u < 0 \} \setminus \p\{ u > 0 \} = \emptyset$, so $\p\{u > 0 \}$ is the only set on which the phase transition occurs.

\begin{rk}
Throughout this paper we assume that $\lambda_2 < \lambda_1$, however all of the analysis is analogous for the opposite case. The value of $\lambda$ at $v=0$ must be chosen so as to make the function lower semi-continuous in $v$.
\end{rk}

In \cite{ACF}, Alt, Caffarelli and Friedman proved a number of properties of minimizers of $J$ in the interior of the domain.
% Some basic properties of a (local) minimizer $u$ proved in \cite{ACF} are, Theorems 1.1-2.3.
% \begin{lemma}\label{basiclemma} As long as $J[u_0]$ is finite, a minimizer $u$ exists, and for any
% such minimizer (or local minimizer):
% \begin{enumerate}
% \item $-\Delta u \leq 0 $ on $\W$, in a distributional sense.  
% \item $\forall \ \W' \Subset \W$, $\forall \ 0 < \alpha < 1$, $u \in C^{0,\alpha}(\W')$.
% \item $\forall \ \W' \Subset \W$, $\{ u \neq 0 \}$ is open in $\W'$.
% \item $\Delta u = 0$ in $\{ u \neq 0 \}$.
% \end{enumerate}
% \end{lemma}
% Using part 4 of Lemma \ref{basiclemma}, and applying the maximum principle in $\{ u > 0\}$ and $\{u < 0 \}$ separately, we may conclude that:
% \begin{lemma}\label{linfty}For almost every $x \in \W$, $A_- \leq u(x) \leq A_+$.
% \end{lemma}
Due to the basic properties of solutions and the maximum principle we know that for almost every $x \in \W$, $$-\max\{u_{0}^{-}(y) : y \in \overline{\W} \} \leq u(x) \leq \max\{u_{0}^{+}(y) : y \in \overline{\W} \}.$$
The minimizer $u$ is H\"{o}lder continuous up to the boundary. The H\"{o}lder exponent is controlled by the Lipschitz constant of $\p \W$.  
This fact is proved in \cite{R}, for the one-phase problem, and the proof for the two-phase problem considered here is identical.
%We here avoid points close to $S$ because the interaction between $N$ and $S$ is a potential difficulty.
% \begin{lemma}\label{holder}
% Let $r_0 > 0$.
% Then $\exists \, \alpha > 0$ such that $u \in C^{0,{\alpha}} (\W_{r_0})$, with $\alpha$ depending on $n$, and $L$.  $\|u\|_{C^\alpha}$ depends on $n, L, M$, $A_+$, $A_-$, and $r_0$.
% \end{lemma}

% \begin{corollary} $\{ u \neq 0 \} \cap N$ is open in $\p \W$.
% \end{corollary}

Additionally, we consider the sense in which Neumann boundary conditions hold for $u$.  Note that $\p_{\nu}u$ may not be
defined pointwise along $\p \W$, and in fact $\nu$ is not defined pointwise.  Therefore, we need an alternate, weak definition of our Neumann boundary conditions, which is as follows:
\begin{DEF}\label{def} We say that a harmonic function $v$ on a   Lipschitz domain $D$ {satisfies Neumann} {boundary conditions
    weakly} along an open set $N \subset \p D$ if    $$\int_D \nab v \cdot \nab \phi \ d x = 0$$ for every $\phi \in
   H^1(D)$, possibly with a boundary condition $\phi=0$ along $\p  D \setminus N$.
\end{DEF}
Note that this concept of Neumann boundary conditions is local, in that the behavior of $v$ away from a neighborhood around $N$ is irrelevant, and if it is proved to hold for a collection of open sets $N_{j} \subset \p D$ such that $\bigcup \limits_j N_{j} = N$, then it
holds on $N$.

We then have \cite{R}:
\begin{lemma} $\p_{\nu}u = 0$ weakly along $N \cap \{ u \neq 0\}$.
\end{lemma}

Finally, note that since $u$ is harmonic where it is nonzero, the maximum principle will prohibit $u$ from being $0$ at
a point $x_0 \in N$ unless $B_r(x_0) \cap \W \cap \{u = 0\} \neq \emptyset$ for all $r > 0$.  Additionally, if $u =0$ in a neighborhood of $x_0$, then obviously $\p_{\nu}u$ is $0$ there.  So the only place in $N$ where the weak Neumann boundary conditions for $u$ might possibly fail is at the free boundary interface itself. In this context the weak maximum principle and the Harnack inequality for harmonic functions are verified in \cite{R}.

% Finally, the following lemmas are also proved in \cite{R}.
% \begin{lemma}\label{maxprinc}
% Let $\W \subset \R^2$ be a bounded, connected Lipschitz domain, and let $\G_D \subset \p \W$ be a measurable set of positive $1$-dimensional Hausdorff measure in $\p \W$.  Let $u \in H^1 (\W)$ satisfy:
% \begin{enumerate}
% \item $\int \limits_\W \nab u \cdot \nab \phi \ d x \geq 0 \qquad \qquad \forall \phi \in P = \{ f \in H^1 | f \geq 0 \mbox{ and } f|_{\G_D} =
% 0 \} $
% \item $\exists \ u_0 \in L^2 (\p \W) \mbox{ such that } u|_{\G_D} = u_0 \geq 0 \mbox{ on } \G_D .$
% \end{enumerate}
% Then $u \geq 0$ in $\W$.
% \end{lemma}

% \begin{lemma} \label{harnack}
% Let $\W$ be a bounded Lipschitz domain in $\R^2$, with Lipschitz constant $L$.  Suppose $\G_N$ is an open subset of $\p \W$.  Let $u$ be a positive harmonic function on $\W$ such that $\p_{\nu}u = 0$ along $\G_N$, in a weak sense.  Then,
% for any $x_0 \in \W \cup \G_N$, for any $r > 0$ such that $B_r (x_0) \cap \p \W \subset \G _N$, $$\sup _{B_\frac{r}2 (x_0)}
% u(x) \leq C \inf _{B_\frac{r}2 (x_0)} u(x)$$ where the constant $C$ depends only on the Lipschitz character of $\p \W$.
% \end{lemma}

Finally, we finish the section with a result from \cite{R} about the regularity of harmonic functions on convex domains.

\begin{lemma} \label{hdharmlemma} Let $\W \subset \R^2$ be a domain such that $\p \W$ is the graph of a convex function $f$.  Suppose $0 \in \W$ and let $r = \mathrm{dist} (0, \p \W)$.  Let $R > 2 r$ and let $D = B_R(0) \cap \W$.  Let $N = B_R \cap \p \W$ and let $S = \overline{\p B_R \cap \W}$.  Let $u$ be a nonnegative harmonic function on D bounded by a constant $A$, with $\p_{\nu}u = 0 $ along $N$.  Then there is an absolute constant $C >0$  such that $|\nab u| \leq C \frac{A}{R}$ on $B_{\frac{R}{2}}$.
\end{lemma}

%%%%%%%%%%%%%%%%%%%%%%%%%%%%%%%%%%%%%%%%%%%%%%%%%%%%%%%%%%%%%%%
%	Main Theorem
%
%%%%%%%%%%%%%%%%%%%%%%%%%%%%%%%%%%%%%%%%%%%%%%%%%%%%%%%%%%%%%%
\section{Main Theorem}

In this section, we present our main result: a gradient bound for minimizers of \eqref{Eqn:EnergyFunc} up to the Neumann boundary on a convex domain in $\R^2$.  To prove this result, we will use a monotonicity lemma the proof of which we adapt from \cite{ACF}. 

%We will follow their proof, modifying it as necessary.

\begin{lemma}\label{monotonicity_formula}
Let $r_{0} > 0$, $x \in N$ with $d(x,\Gamma) < r_{0}$ and suppose $d(x,S) \geq r_{0}$. Set 
$$\phi (r) = \frac{1}{r^{4}} \int_{B_{r} \cap \Omega} |\nabla u^{+}|^{2} \ d x \cdot \int_{B_{r} \cap \Omega} |\nabla u^{-}|^{2} \ d x.$$
If $u \in C(B_{r} \cap \Omega) \cap H^{1}(B_{r} \cap \Omega)$ satisfies $\p_{\nu}u = 0$ on $N$ and $\Delta u = 0$ in $B_{r} \cap \Omega \setminus \{u=0\}$, then $\phi '(r) \geq 0$.
\end{lemma}
\begin{proof}

\begin{figure}[h]
\centering
\includegraphics[scale=0.5]{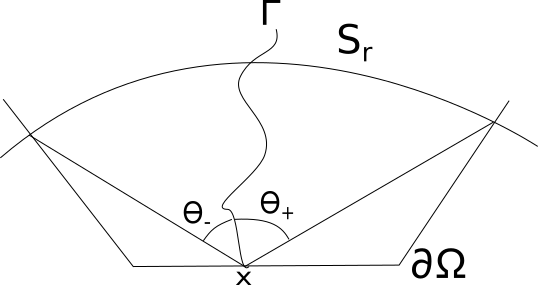}
\caption{Domain}
\label{fig:domain}
\end{figure}

Let $x \in N$, $B_{r} = B_{r}(x)$, $S_{r} = S_{r}(x)$ and $\theta_{+},\theta_{-}$ be defined as in Figure \ref{fig:domain}. Throughout this proof we will let $(r,\theta)$ denote polar coordinates centered at $x$. It follows from convexity that $\theta_{+} + \theta_{-} \leq \pi$ (see Figure \ref{fig:domain}). 

The solution to the eigenvalue problem
\begin{equation*}
\frac{-1}{r^{2}} \frac{d^{2}f_{\pm}}{d\theta^{2}} = \lambda_{\pm}f_{\pm}, ~f_{\pm}(\theta_{\pm}) = f_{\pm}'(0)=0
\end{equation*}
is given by
$$ f_{\pm}=C_{\pm}\cos\left(\sqrt{\lambda_{\pm}}r\theta\right)$$
with 
$$ 
\sqrt{\lambda_{\pm}}r\theta_{\pm}=\frac{\pi}{2}.
$$
By scaling we may assume, without loss of generality, that $r=1$.  Consequently, since $\theta_{+}^{-1} + \theta_{-}^{-1} \geq \theta_{+}^{-1} + (\pi - \theta_{+})^{-1}\geq \frac{4}{\pi}$, it follows that
\begin{equation}\label{Eqn:LambdaBound}
\sqrt{\lambda_{+}} + \sqrt{\lambda_{-}} = \frac{\pi}{2\theta_{+}} + \frac{\pi}{2\theta_{-}} \geq 2.
\end{equation}
% $$f_{+}(\theta) = C_{1}^{+}\sin\sqrt{\lambda_{+}}r\theta + C_{2}^{+}\cos\sqrt{\lambda_{+}}r\theta.$$
% Further, $f_{+}'(0) = 0$ implies $C_{1}^{+} = 0$. So,
% $$f_{+}(\theta) = C_{2}^{+}\cos\sqrt{\lambda_{+}}r\theta.$$
% Further, $f_{+}(\theta_{+}) = 0$ gives $\cos\sqrt{\lambda_{+}}r\theta_{+} = 0$. It then follows that
% $$\sqrt{\lambda_{+}}r\theta_{+} = \frac{\pi}{2}.$$
% Now consider
% $$\frac{-1}{r^{2}} \frac{d^{2}f_{-}}{d\theta^{2}} = \lambda_{-}f_{-}, ~f_{-}(\theta_{+}) = f_{-}'(\theta_{+} + \theta_{-})=0.$$
% Equivalent boundary conditions are given by $f_{-}(\theta_{-}) = f_{-}'(0) = 0$. Thus, we have, as before,

% $$
% \sqrt{\lambda_+}+\sqrt{\lambda_-}=
% $$
% $$$$
% and
% $$(\theta_{+}^{-1} + (\pi - \theta_{+})^{-1}))' = -\theta_{+}^{-2} + (\pi - \theta_{+})^{-1})^{-2}.$$
% So, $(\theta_{+}^{-1} + (\pi - \theta_{+})^{-1})^{-1})' = 0$ implies $(\pi - \theta_{+})^{-1})^{-2} = \theta_{+}^{-2}$ which yields $\pi^{2} - 2\pi\theta_{+} + \theta_{+}^{2} = \theta_{+}^{2}$. Note that there is no maximum. Thus, $\theta_{+} = \frac{\pi}{2}$ must be a local minimum. Hence, $\theta_{+}^{-1} + \theta_{-}^{-1} \geq \frac{4}{\pi}$. Therefore,

Now let $\W^{+}$ and $\W^{-}$ be the support of $u^{+}$ and $u^{-}$ respectively on $S_{1} \cap \Omega$. Applying the Rayleigh quotient gives $$\lambda_{\pm} = \inf_{v \in H_{0}^{1}(\W^{\pm})} \frac{\int_{\W^{\pm}} (\p_{\theta}v)^{2} \ d \s}{\int_{\W^{\pm}} v^{2} \ d \s} \leq \frac{\int_{S_{1} \cap \Omega} (\p_{\theta}u^{\pm})^{2} \ d \s}{\int_{S_{1} \cap \Omega} (u^{\pm})^{2} \ d \s},$$ where $d\sigma$ is the line element.
Consequently, 
$$\sqrt{\int_{S_{1} \cap \Omega} (\p_{\theta}u^{\pm})^{2} \ d \s} \geq \sqrt{\lambda_{\pm}}\sqrt{\int_{S_{1} \cap \Omega} (u^{\pm})^{2} \ d \s}$$
and therefore it follows that
\begin{align}
\int_{S_{1} \cap \Omega} |\nabla u^{\pm}|^{2} \ d \s &= \int_{S_{1} \cap \Omega} (\p_{r}u^{\pm})^{2} + (\p_{\theta}u^{\pm})^{2} \ d \s \nonumber \\
&\geq 2\sqrt{\int_{S_{1} \cap \Omega} (\p_{r}u^{\pm})^{2} \ d \s \cdot \int_{S_{1} \cap \Omega} (\p_{\theta}u^{\pm})^{2} \ d \s} \nonumber \\
&\geq 2\sqrt{\lambda_{\pm}}\sqrt{\int_{S_{1} \cap \Omega} (\p_{r}u^{\pm})^{2} \ d \s \cdot \int_{S_{1} \cap \Omega} (u^{\pm})^{2} \ d \s} \nonumber \\
&\geq 2\sqrt{\lambda_{\pm}}\int_{S_{1} \cap \Omega} |u^{\pm}\p_{r}u^{\pm}| \ d \s. \label{Eqn:RayleighBound}
\end{align}

Finally, differentiating it follows that
\begin{align*}
\begin{split}
\phi '(r) = {}&\frac{-4}{r^{5}} \left[ \int_{B_{r} \cap \Omega} |\nabla u^{+}|^{2} \ d x \cdot \int_{B_{r} \cap \Omega} |\nabla u^{-}|^{2} \ d x \right] + \frac{1}{r^{4}} \left[ \int_{S_{r} \cap \Omega} |\nabla u^{+}|^{2} \ d \s \cdot \int_{B_{r} \cap \Omega} |\nabla u^{-}|^{2} \ d x \right] \\ 
{}& +\frac{1}{r^{4}} \left[ \int_{B_{r} \cap \Omega} |\nabla u^{+}|^{2} \ d x \cdot \int_{S_{r} \cap \Omega} |\nabla u^{-}|^{2} \ d \s \right].
\end{split}
\end{align*}
But, since in $\Omega$ either $u^{\pm}=0$ or $\Delta u^{\pm}=0$, it follows from integration by parts and the Neumann boundary conditions that
\begin{align*}
\int_{B_{r} \cap \Omega} |\nabla u^{\pm}|^{2} \ d x = \int_{S_{r} \cap \Omega} u^{\pm}\p_{r}u^{\pm} \ d \s,
\end{align*}
where we have used the fact that $\p_{\nu}u^{\pm} = \p_{r}u^{\pm}$ on $S_{r}$. So, using the bounds (\ref{Eqn:Lambda}) and (\ref{Eqn:RayleighBound}) it follows that
\begin{align*}
\begin{split}
\phi '(1) = {}&-4  \int_{S_{1} \cap \Omega} u^{+}\p_{r}u^{+} \ d \s \cdot \int_{S_{1} \cap \Omega} u^{-}\p_{r}u^{-} \ d \s  +  \int_{S_{1} \cap \Omega} |\nabla u^{+}|^{2} \ d \s \cdot \int_{S_{1} \cap \Omega} u^{-}\p_{r}u^{-} \ d \s  \\ 
{}&  + \int_{S_{1} \cap \Omega} u^{+}\p_{r}u^{+} \ d \s \cdot \int_{S_{1} \cap \Omega} |\nabla u^{-}|^{2} \ d \s \\
\geq {}&-4\int_{S_{1} \cap \Omega} |u^{+}\p_{r}u^{+}| \ d \s \cdot \int_{S_{1} \cap \Omega} |u^{-}\p_{r}u^{-}| \ d \s + 2\sqrt{\lambda_{+}}\int_{S_{1} \cap \Omega} |u^{+}\p_{r}u^{+}| \ d \s \cdot \int_{S_{1} \cap \Omega} |u^{-}\p_{r}u^{-}| \ d \s \\
{}& +2\sqrt{\lambda_{-}}\int_{S_{1} \cap \Omega} |u^{+}\p_{r}u^{+}| \ d \s \cdot \int_{S_{1} \cap \Omega} |u^{-}\p_{r}u^{-}| \ d \s\\
={}& [-4 + 2(\sqrt{\lambda_{+}} + \sqrt{\lambda_{-}})]\int_{S_{1} \cap \Omega} u^{+}\p_{r}u^{+} \ d \s \cdot \int_{S_{1} \cap \Omega} u^{-}\p_{r}u^{-} \ d \s \geq 0.
\end{split}
\end{align*}
% However, previous computations showed that
% \begin{align*}
% \begin{split}
% {}& -4\int_{S_{1} \cap \Omega} |u^{+}\p_{r}u^{+}| \ \id \s \cdot \int_{S_{1} \cap \Omega} |u^{-}\p_{r}u^{-}| \ \id \s + \int_{S_{1} \cap \Omega} |\nabla u^{+}|^{2} \ \id \s \cdot \int_{S_{1} \cap \Omega} u^{-}\p_{r}u^{-} \ \id \s +\\ 
% {}& \int_{S_{1} \cap \Omega} u^{+}\p_{r}u^{+} \ \id \s \cdot \int_{S_{1} \cap \Omega} |\nabla u^{-}|^{2} \ \id \s\\ 
% \geq {}&-4\int_{S_{1} \cap \Omega} |u^{+}\p_{r}u^{+}| \ \id \s \cdot \int_{S_{1} \cap \Omega} |u^{-}\p_{r}u^{-}| \ \id \s + 2\sqrt{\lambda_{+}}\int_{S_{1} \cap \Omega} |u^{+}\p_{r}u^{+}| \ \id \s \cdot \int_{S_{1} \cap \Omega} |u^{-}\p_{r}u^{-}| \ \id \s +\\
% {}& 2\sqrt{\lambda_{-}}\int_{S_{1} \cap \Omega} |u^{+}\p_{r}u^{+}| \ \id \s \cdot \int_{S_{1} \cap \Omega} |u^{-}\p_{r}u^{-}| \ \id \s\\
% \end{split}
% \\
% {}&= [-4 + 2(\sqrt{\lambda_{+}} + \sqrt{\lambda_{-}})]\int_{S_{1} \cap \Omega} u^{+}\p_{r}u^{+} \ \id \s \cdot \int_{S_{1} \cap \Omega} u^{-}\p_{r}u^{-} \ \id \s\\
% {}&\geq 0
% \end{align*}
% as $\sqrt{\lambda_{+}} + \sqrt{\lambda_{-}} \geq 2$. 
Therefore, we can conclude, after rescaling, that
$$\phi '(r) \geq 0.$$
\end{proof}

% Before proceeding we recall a theorem from \cite{LSW} regarding the comparability of Green's functions for uniformly elliptic operators with the same ellipticity constant. \textcolor{red}{We use this theorem once inside of a proof. Lets just remove this theorem and cite it within the proof.}

% \begin{theorem}\label{Green_function_LSW}\cite{LSW}
% Let $G(x,y)$ and $\tilde{G}(x,y)$ be the Green's functions for any uniformly elliptic operators $L$ and $\tilde{L}$ with ellipticity constant $\theta$ on a ball $\Sigma \subset \R^n$. Then, for any compact $\Sigma' \subset \Sigma$ there is a constant $C = C(\Sigma',\Sigma,\theta)$ we have
% $$C^{-1} \tilde{G}(x,y) \leq G(x,y) \leq C \tilde{G}(x,y) \qquad (x,y \in \Sigma').$$

% \end{theorem}

% It should be noted that the above theorem can be readily extended to any simply connected domain $\W$ which can be smoothly mapped onto a ball.

\begin{lemma}\label{harmonic_measure_bound}
	We have $\Delta u(B_{\frac{r}{2}} \cap \Omega) \leq Cr$.
\end{lemma}
\begin{proof}
Let $v$ be the solution of
\begin{align*}
\Delta v &= 0 \ \text{in} \ B_{r}\\
v|_{S_{r}} &= u.
\end{align*}
Since $u$ is a minimizer:
\begin{align*}
\int_{B_{r} \cap \Omega} |\nabla u|^{2} \ d x - \int_{B_{r} \cap \Omega} |\nabla v|^{2} \ d x
&\leq |\lambda_{1}^{2} - \lambda_{2}^{2}| \int_{B_{r} \cap \Omega} Q^{2}(x) \ d x\\
&\leq |\lambda_{1}^{2} - \lambda_{2}^{2}| \|Q^{2}\|_{L^{\infty}(B_{r} \cap \Omega)} |B_{r} \cap \Omega|\\
&\leq Cr^{2}.
\end{align*}
However,
\begin{equation*}
\int_{B_{r} \cap \Omega} |\nabla u|^{2} - |\nabla v|^{2} \ d x = \int_{B_{r} \cap \Omega} \nabla(u-v) \cdot \nabla(u-v) \ d x + \int_{B_{r} \cap \Omega} 2\nabla(u-v) \cdot \nabla v \ d x.
\end{equation*}
Notice that
\begin{align*}
\int_{B_{r} \cap \Omega} 2\nabla(u-v) \cdot \nabla v \ d x &= -\int_{B_{r} \cap \Omega} 2(u-v)\Delta v \ d x + \int_{S_{r} \cap \Omega} 2(u-v)\p_{\nu}v \ d x + \int_{B_{r} \cap \p\Omega} 2(u-v)\p_{\nu} v \ d x\\
&= 0,
\end{align*}
as $\Delta v = 0$ in $B_{r}$, $u=v$ on $S_{r}$ and $\p_{\nu}v = 0$ on $\p\Omega$. So,
\begin{align*}
\int_{B_{r} \cap \Omega} |\nabla u|^{2} - |\nabla v|^{2} \ d x = {}& \int_{B_{r} \cap \Omega} \nabla(u-v) \cdot \nabla(u-v) \ d x\\
= {}& \int_{B_{r} \cap \Omega} |\nabla u|^{2} - \nabla u \cdot \nabla v - \nabla v \cdot \nabla u + |\nabla v|^{2} \ d x\\
\begin{split}
= {}&\int_{B_{r} \cap \Omega} |\nabla u|^{2} - \nabla v \cdot \nabla u \ d x + \int_{B_{r} \cap \Omega} u \Delta v - v \Delta v \ d x + \int_{S_{r} \cap \Omega} v\p_{\nu}v - u\p_{\nu}v \ d x +\\ {}& \int_{B_{r} \cap \p\Omega} v\p_{\nu}v - u\p_{\nu}v \ d x
\end{split}\\
= {}& \int_{B_{r} \cap \Omega} |\nabla u|^{2} - \nabla v \cdot \nabla u \ d x\\
= {}&\int_{B_{r} \cap \Omega} \nabla(u-v) \cdot \nabla u \ d x\\
= {}&\int_{B_{r} \cap \Omega} (v-u)\Delta u \ d x + \int_{S_{r} \cap \Omega} (v-u)\p_{\nu}u \ d \s + \int_{B_{r} \cap \p\Omega} (v-u)\p_{\nu}u \ d x\\
= {}&\int_{B_{r} \cap \Omega} (v-u)\Delta u \ d x\\
= {}&\int_{B_{r} \cap \Omega} v\Delta u \ d x,
\end{align*}
as $\Delta u$ is a measure supported on $\{u=0\}$. Therefore,
$$\int_{B_{\frac{r}{2}} \cap \Omega \cap \{u=0\}} v\Delta u \ \id x \leq Cr^{2}.$$

Since $\Omega$ is a Lipschitz domain, there is a bilipschitz map
$$F: B_{r} \cap \Omega \to B_{r}^{+}.$$ Define the operator $L$ by $Lv=\p_i(a^{ij}\p_j v),$ where $a^{ij}(x)=|\mathrm{det}(\nab F^{-1})| (\nab F)^T\nab F,$ and let $\tilde{u}=u\circ F^{-1}$, $\tilde{v}=v\circ F^{-1}$. We will show that $\tilde{v}$ satisfies $L\tilde{v} = 0 $ in $B_r^+$ and $\tilde{v}=\tilde{u}$ on $\p B_r^+\cap \{x_n > 0\}$.  We then use an even reflection to find a solution to $L\tilde{v}=0$ in $B_r$.  Note the coefficients of $L$ are necessarily bounded and measurable.  Therefore, there is a Green's function $\tilde{G}$ associated to this operator, and, as proved in \cite{LSW}, if $G$ is the standard Green's function on $B_r$, then there are positive constants $c$ and $C$ so that $cG \leq \tilde{G}\leq CG$ on $B_r$.  Additionally, define the function $H$ on $B_r\cap \Omega$ by $H(x,y)=\tilde{G}(F(x),F(y))$.  

Note next we have, as in \cite{LSW}, that
	$$\tilde{u}(x^{0}) = \tilde{v}(x^{0}) - \int_{B_{r}} \tilde{G}_{x^{0}}L\tilde{u}(y) \ d y.$$
	%Let
	%$$\tilde{v}(x) \coloneqq \int_{B_{t}} \tilde{P_{x}}(y)\tilde{u}(y) \qquad \text{Justified?}$$
	%where $\tilde{u}(y) = u \circ F^{-1}(y)$, $\tilde{P_{x}}(y) = P_{x} \circ F^{-1}(y)$, etc.\\
	%Recall that $\tilde{P_{x^{0}}}$ is not necessarily the normal derivative of the Green's function.\\
Letting $x^{0} \in \{\tilde{u}=0\}$ it follows that
	$$\tilde{v}(x^{0}) = \int_{B_{r}} \tilde{G}_{x^{0}}(y)L\tilde{u}(y) \ d y.$$	
	%\begin{align*}
	%0 &= \tilde{u}(x^{0})\\
	%&= \int_{S_{t}} \tilde{P_{x}}(y)\tilde{u}(y) - \int_{B_{t}} \tilde{G_{x^{0}}}(y)L\tilde{u}(y)
	%\end{align*}
Let $V = F(B_{\frac{r}{2}} \cap \Omega)$. Then, $V \subset B_{r}$ and $F$ being bilipschitz together imply that
	$$Cr^{2} \geq \int_{B_{\frac{r}{2}} \cap \Omega} v\Delta u \ d x = \int_{V} \tilde{v}L\tilde{u} \ d x = c\int_{V} \left( \int_{B_{r}} \tilde{G}_{x^{0}}(y)L\tilde{u}(y) \ d y \right)L\tilde{u}(x) \ d x.$$	
	%$$\int_{V}\left( \int_{B_{t}} \tilde{G_{x}}(y)L\tilde{u}(y) \right)L\tilde{u}(x) \leq Cr^{2}$$
Notice that $\tilde{G}_{x^{0}}(y) \geq c > 0$ for $x,y \in V$ \cite{LSW}. It follows that 
	\begin{align*}
	Cr^{2} &\geq \int_{V}\left( \int_{B_{r}} \tilde{G}_{x^{0}}(y)L\tilde{u}(y) \ d y \right)L\tilde{u}(x) \ d x\\
	&\geq c\int_{V}\left( \int_{B_{r}} L\tilde{u}(y) \ d y \right)L\tilde{u}(x) \ d x\\
	&= c\int_{V} (L\tilde{u}(B_{r}))L\tilde{u}(x) \ d x\\
	&= c(L\tilde{u}(B_{r}))\int_{V} L\tilde{u}(x) \ d x\\
	&= cL\tilde{u}(B_{r})L\tilde{u}(V)\\
	&\geq c(L\tilde{u}(V))^{2}.
	\end{align*}
	Therefore, $L\tilde{u}(V) \leq Cr$. Since $F$ is bilipschitz, it follows that $\Delta u(B_{\frac{r}{2}} \cap \Omega) \leq Cr$.
\end{proof}

\begin{lemma}\label{upper_estimate_on_averages}
	Let $\max\{\lambda_{1}^{2},\lambda_{2}^{2}\} = \ell_{1}$. If $B_{r}$ has center in $\{u=0\}$, then there is a positive constant $C = C(q_{2}, \ell_{1})$ such that
	$$\frac{1}{r} \left| \fint_{S_{r} \cap \Omega} u \ d \s \right| \leq C.$$
\end{lemma}
\begin{proof}
	Assume that the center of $B_{r}$ is the origin. Using the notation of the previous lemma, and assuming $F(0)=0$, we have that
	\begin{equation*}
	0 = u(0)= \tilde{u}(0)= \int_{B_{r}} \tilde{G}_{0}(y)L\tilde{u}(y) \ d y - \tilde{v}(0),
	\end{equation*}
	with $\tilde{v}$ defined as in the proof of the previous lemma, and $\tilde{G}_0$ the Green's function centered at the origin. Let
	$$I \coloneqq \int_{B_{r}}\tilde{G}_{0}(y)L u(y) \ d y.$$
	Then, as before, there are constants $c, C >0$ so that 
	$$c \int_{B_{r}} {G_{0}}(y)L\tilde{u}(y) \ dy \leq  I \leq C \int_{B_{r}} {G_{0}}(y)L\tilde{u}(y) \ dy.$$
	%Let
	%$$h(r) \coloneqq r^{n-1} \int_{S_{1}} L\tilde{u}(r\xi) ~d\mathscr{H}^{n-1}(\xi)$$
Next, using radial symmetry of the standard Green's function it follows that
	\begin{align*}
	\int_{B_{r}} {G_{0}}(y)L\tilde{u}(y) 
	&= \int_{0}^{r} \int_{B_{s}} {G_{0}}(s,\theta)L\tilde{u}(s,\theta) s \ d\theta ds\\
    & = \int_0^r s g(s)\int_{B_{s}}L\tilde{u}(s,\theta) \ d\theta ds.
	\end{align*}
	Here, $g(s) = G_{0}(s,\theta) = -\log\left(\frac{s}{r}\right)$ for $0 \leq \theta \leq 2\pi$. Let
	$$h(s) = s\int_{S_{1}} L\tilde{u}(s,\theta) \ d\theta.$$
	Then,
	\begin{align*}
\int_{0}^{r} {G}(s,\theta) L\tilde{u}(s,\theta) s \ d\theta ds & = \int_{0}^{r} s g(s) \int_{S_{1}} L\tilde{u}(s,\theta) \ d\theta ds\\
	&= \int_{0}^{r} -s \log \left(\frac{s}{r} \right) \int_{S_{1}} L\tilde{u}(s,\theta) \ d\theta ds\\
	&= \int_{0}^{r} -\log \left(\frac{s}{r} \right) h(s) \ ds\\
	&= \int_{0}^{r} -\log \left(\frac{s}{r} \right) \frac{d}{ds}\left( \int_{0}^{s} h(t) \ dt \right) \ ds\\
	&=  \left[ -\log \left(\frac{s}{r} \right) \int_{0}^{s} h(t) \ dt \right]_{0}^{r} - C\int_{0}^{r} -\frac{1}{s} \int_{0}^{s} h(t) \ dt ds\\
	&= 0 + \lim_{s \to 0} \left(\log \left( \frac{s}{r} \right) \int_{0}^{s} h(t) \ dt \right) + \int_{0}^{r} \frac{1}{s} \int_{0}^{s} h(t) \ dt ds\\
	&\leq \int_{0}^{r} \frac{1}{s} \int_{0}^{s} t  \int_{S_{1}} L\tilde{u}(t,\theta) \ d\theta dt ds\\
	&\leq \int_{0}^{r} \frac{1}{s}Cs \ ds\\
	&= C\int_{0}^{r} \ ds\\
	&= Cr,
	\end{align*}
    where we have used the previous lemma to estimate the integral of $L\tilde{u}$ over $B_r$.  Note that 
    $\lim_{s\to 0}(\log(\frac{s}{r})\int_0^s h(t)dt)$ is bounded above by $0$ because the logarithmic term is negative for small $s$ and the function $h(t)$ is nonnegative.  
    
From this estimate we may conclude that $$\int_{B_{r}} \tilde{G}_{0}(y)L\tilde{u}(y) \ dy \leq Cr$$ as well.  Notice that $\tilde{v}-\tilde{u}$ satisfies $L(\tilde{v}-\tilde{u})=-L(\tilde{u})$ and $\tilde{u}-\tilde{v}=0$ on $S_r(0)$.  Therefore, we can conclude that $\tilde{v}(0)-\tilde{u}(0)=\int_{B_r(0)}\tilde{G}_0(x)(-L(\tilde{u}(x))) \ dx.$ Since $u(0)=0$, we may conclude that $\tilde{v}(0)=-\int_{B_r(0)}\tilde{G}_0(x)L(\tilde{u}(x)) \ dx.$ 

Now, since $F$ is bilipschitz, $v$ and $u$ have the same boundary conditions on $B_r\cap\Omega$ and $v$ is harmonic there, it follows that
 	\begin{equation*}
 	\fint_{S_{r} \cap \Omega} u \ d\s = v(0)=\tilde{v}(0)=\int_{B_r}G_0 L \tilde{u} \ dx\leq Cr.
    \end{equation*}
Since $u$ is subharmonic in $B_{r} \cap \Omega$ and $u(0) = 0$, it follows by the mean-value property for subharmonic functions that
	$$\fint_{S_{r} \cap \Omega} u \ d\s \geq 0.$$
	Therefore,
	\begin{equation*}
	\frac{1}{r} \left| \fint_{S_{r} \cap \Omega} u \ d\s \right| = \frac{1}{r}\fint_{S_{r} \cap \Omega} u \ d\s \leq C.
	\end{equation*}
\end{proof}

Now we come to our main result:

\begin{theorem}\label{Lipschitz_continuity}
	Let $r_{0}>0$ and define $\Omega_{r_{0}} \coloneqq \{x \in \Omega : \mathrm{dist}(x,S)>r_{0} \}$. Then, there is a constant $C$ such that if $u$ is a minimizer of the functional $J$, then for almost every $x \in \Omega_{r_{0}}$ we have
	$$|\nabla u(x)| \leq C.$$
\end{theorem}

\begin{proof} Let $x \in \W_{r_0}$.  We know from \cite{ACF} that there is a $C > 0$ so that if $d(x,\p\W) \geq r_0$, then $|\nab u(x)| \leq C$. Moreover, if $d(x, \p\W) > d(x,\G),$ the argument in \cite{ACF} will also go through. On the other hand, if $d(x,\G) \geq r_0$, then standard interior harmonic regularity or Lemma \ref{hdharmlemma} implies the desired gradient bound.  So we are primarily interested in the case where $d(x,N) \leq d(x,\G) < r_0$.  Following the argument in \cite{R}, Theorem 2, it suffices to control $|\nab u|$ on $\p \W \cap B_{r_0}(\G)$.  So let $x \in \p \W \cap B_{r_0}(\G)$.  We follow the argument in \cite{ACF}, second proof of Theorem 5.3.  Since the Harnack inequality (\cite{R}), monotonicity formula (Lemma \ref{monotonicity_formula}), and upper estimate on averages (Lemma \ref{upper_estimate_on_averages}) all hold in our context, the argument proceeds in the same fashion.  Note that convexity implies that the use of polar coordinates in the proof will work as desired.
\end{proof}

\section{Numerical Scheme}

In this section we present our numerical scheme. Our approach is to apply a gradient flow to a version of $J$. In the modification, the phase term in the functional is approximated by a regularized transition layer. This approximation of the energy allows us to define the gradient flow in the classical sense. We then implement the gradient flow using a finite difference scheme. The free boundary is then recovered as the appropriate contour of the data.

\subsection{Relaxed Functional and Gradient Flow}
Define a sequence of width $\varepsilon$ transition layers $\varphi_{\varepsilon}\in C^{1,1}(\mathbb{R};[0,1])$  as a class of functions satisfying the following properties:
\begin{enumerate}
\item $\varphi_1 \in C^{1,1}(\mathbb{R};[0,1])$,
\item $\varphi_1(v)=\lambda_1^2$ if $v\geq 1$,
\item $\varphi_1(v)=\lambda_2^2$ if $v\leq  0$,
\item $\varphi_1^{\prime}(v)\geq 0$,
\item $\varphi_{\varepsilon}=\varphi_1\left(\frac{v}{\varepsilon}\right)$.
\end{enumerate}
Clearly $\varphi_{\epsilon}(v)$ converges pointwise to $\lambda^2(v)$ as $\varepsilon\rightarrow 0$. Moreover, since $\varphi^{\prime}_{\varepsilon}$ is a sequence of $C^{0,1}$ functions compactly supported on $[0,\varepsilon]$ satisfying 
\begin{equation*}
\int_{0}^{\varepsilon} \varphi_{\varepsilon}^{\prime}(v)dv=\frac{1}{\varepsilon}\int_{0}^{\varepsilon} \varphi_{1}^{\prime}\left(\frac{v}{\varepsilon}\right)dv=\int_0^1 \varphi_1^{\prime}(v)dv=\lambda_1^2-\lambda_2^2,
\end{equation*}
it follows that in the sense of distributions $\varphi_{\varepsilon}^{\prime}\rightarrow \left(\lambda_2^2-\lambda_1^2\right)\delta(v)$, where $\delta(v)$ denotes the Dirac delta function. The relaxed functional $J_{\varepsilon}:K\mapsto \mathbb{R}^+$ is then defined by
\begin{equation}\label{eqn:RelaxedFunctional}
J_{\varepsilon}[v]=\int_{\Omega}\left(\left| \nabla v \right|^2+Q^2(x)\varphi_{\varepsilon}(v)\right) \id x.
\end{equation}

Since $J_{\varepsilon}$ is convex in $\nabla v$, it follows from the direct method of the calculus of variations that $J_{\varepsilon}$ has a minimum in $K$ \cite{jost1998calculus}. Moreover, since $\varphi_{\varepsilon}\in C^{1,1}$ it follows that minimizers $u_{\varepsilon}$ of $J_{\varepsilon}$ will satisfy the following nonlinear Poisson equation:
\begin{equation}\label{Eqn:RelaxedStrongEL}
\begin{cases}
\displaystyle{2\Delta u_{\varepsilon}=Q(x)\varphi^{\prime}(u_{\varepsilon})}\\
\displaystyle{\left.  \partial_{\nu}u_{\varepsilon}\right|_N=0}\\
\displaystyle{\left. u_{\varepsilon} \right|_{S}=u_0}
\end{cases},
\end{equation}
where the normal derivative $\partial_{\nu}u_{\varepsilon}$ is interpreted in the weak sense; see Definition 1.

We now establish that minimizers of $J_{\varepsilon}$ converge up to a subsequence to a minimizer of $J$ with respect to the $H^1$ norm. The failure to improve from convergence of subsequences to full convergence results from the possible non-uniqueness of minimizers. In practice, however, we expect the minimizers of $J_{\varepsilon}$ will be generated using a consistent numerical scheme and hence the minimizers of the relaxed functional will strongly converge in $H^1$ to the minimizer of $J$.

\begin{theorem} \label{Thm:StrongConvergence}
Let $u_{\varepsilon}\in K$ be minimizers of $J_{\varepsilon}$. Then, there exists $u\in K$ minimizing $J$ such that $J_{\varepsilon}[u_{\varepsilon}]\rightarrow J[u]$ and there exists a subsequence $u_{\varepsilon_k}$ such that $u_{\varepsilon_k}\stackrel{H^1}{\rightarrow}u$.
\end{theorem}
\begin{proof}
Let $u_{\varepsilon}\in K$ be minimizers of $J_{\varepsilon}$ and $\bar{u}\in K$ be a minimizer of $J$. Since $\varphi_{\varepsilon}$ is a monotone increasing sequence of functions as $\varepsilon \rightarrow 0$ it follows for all $\varepsilon>0$ that $J_{\varepsilon}[\bar{u}]\leq J[\bar{u}]<\infty$. Consequently, for $\varepsilon^{\prime}<\varepsilon$ it follows that
\begin{equation*}
J_{\varepsilon}[u_{\varepsilon}] \leq J_{\varepsilon}[u_{\varepsilon^{\prime}}]\leq J_{\varepsilon^{\prime}}[u_{\varepsilon^{\prime}}]<J[\bar{u}]
\end{equation*}
and thus $J_{\varepsilon}[u_{\varepsilon}]$ is a bounded monotone increasing sequence as $\varepsilon \rightarrow 0$ and hence converges. Moreover, it follows from this estimate and Poincare's inequality \cite{adams2003sobolev} that $u_{\varepsilon}$ is bounded in the $H^1$ norm and hence there exists $u^*\in K$ and a subsequence $u_{\varepsilon_k}$ such that $u_{\varepsilon}\stackrel{L^2}{\rightarrow}u^*$, $u_{\varepsilon}\stackrel{H^1}{\rightharpoonup}u^*$, and $u_{\varepsilon}\rightarrow u^*$ pointwise. Therefore,
\begin{equation*}
J[\bar{u}]\geq \lim_{\varepsilon \rightarrow 0}J_{\varepsilon}[u_{\varepsilon}]=\lim_{k\rightarrow \infty} J_{\varepsilon_k}[u_{\varepsilon_k}]=J[u^*]\geq J[\bar{u}].
\end{equation*}
Since the lower and upper bounds in the above chain of inequalities are equal it follows that all of the inequalities are in fact equalities and hence
\begin{equation*}
\lim_{\varepsilon \rightarrow 0}J_{\varepsilon}[u_{\varepsilon}]=J[u^*]=J[\bar{u}]=\min_{v\in K}J[v],
\end{equation*}
and therefore $u^*$ is a minimizer of $J$ as well. 

Finally, we show strong convergence of the subsequence. Since $\nabla u_{\varepsilon}\rightharpoonup \nabla u^*$ it follows that
\begin{align*}
 \lim_{k\rightarrow \infty}\left\| \nabla u^*-\nabla u_{\varepsilon_k} \right\|_{L^2}^2&=\lim_{k\rightarrow \infty}\left(\|\nabla u^*\|_{L^2}^2-2\langle \nabla u^*, \nabla u_{\varepsilon_k} \rangle+\|\nabla u_{\varepsilon_k}\|_{L^2}^2\right)\\
&=\lim_{k\rightarrow \infty}\left(\|\nabla u^*\|_{L^2}^2-2\langle \nabla u^*, \nabla u_{\varepsilon_k} \rangle+\|\nabla u_{\varepsilon_k}\|_{L^2}^2\right)\\
&\,\,\,\,+\lim_{k\rightarrow \infty}\left(\int_{\Omega}Q^2(x)\left(\varphi_{\varepsilon_k}(u_{\varepsilon_k})-\varphi_{\varepsilon_k}(u_{\varepsilon_k})\right)dx\right)\\
&=-J[u^*]+J[u^*]=0.
\end{align*}
\end{proof}

The next result ensures that a convergent sequence of local minimizers of $J_{\varepsilon}$ converges to a local minimizer of $J$. 

\begin{theorem} \label{Thm:LocalConvergence} Let $u_\varepsilon$ be a sequence of local minimizers of $J_{\varepsilon}$ in the the sense that there exists uniform $\delta>0$ such that $J_{\varepsilon}[u_{\varepsilon}]<J_{\varepsilon}[v]$ for all $v$ satisfying $\|u_{\varepsilon}-v\|_{H^1}<\delta$. If $u_{\varepsilon_k}\stackrel{H^1}{\rightarrow}u$, then $u$ is a local minimizer of $J$. 
\end{theorem}
\begin{proof}
Suppose $v\in K$ satisfies $\|v-u\|_{H^1}<\delta/2$. Since $u_{\varepsilon}\stackrel{H^1}{\rightarrow}u$ there exists $\varepsilon^{\prime}$ such that $\varepsilon<\varepsilon^{\prime}$ implies $\|u-u_{\varepsilon}\|<\delta/2$. Consequently, applying the triangle inequality it follows that, upon passing to a subsequence $u_{\varepsilon_k}$ to ensure pointwise convergence, that
\begin{equation*}
J[u]=\lim_{k\rightarrow \infty}J_{\varepsilon_k}[u_{\varepsilon_k}]\leq \lim_{k\rightarrow \infty}J_{\varepsilon_k}[v]=J[v].
\end{equation*}
\end{proof}

Solutions to Eq. (\ref{Eqn:RelaxedStrongEL}) can be generated by applying a gradient flow to $J_{\varepsilon}$. Namely, we consider solutions $v:\mathbb{R}^+\times \Omega\mapsto \mathbb{R}$ to the following reaction diffusion equation:
\begin{equation}\label{Eqn:GradientFlow}
\begin{cases}
v_t=2\Delta v-Q(x)\varphi^{\prime}_{\varepsilon}(v)\\
\left.\partial_{\nu}v\right|_N=0\\
\left.v\right|_S=u_0\\
v(0,x)=v_0(x)
\end{cases},
\end{equation}
where $v_0\in K$. If we consider \eqref{Eqn:GradientFlow} as an infinite dimensional dynamical system, we find that $J_{\varepsilon}$ is a Lyapunov function and consequently solutions $v(x,t)$ satisfy
\begin{equation}
\lim_{t\rightarrow \infty}v(x,t)=u_{\varepsilon}(x)\in K,
\end{equation}
where $u_{\varepsilon}$ is a (local) minimizer of $J_{\varepsilon}$ and hence is a solution to Eq. (\ref{Eqn:RelaxedStrongEL}); see \cite{robinson2001infinite} Chapter 11.

\begin{rk} By Theorem \ref{Thm:StrongConvergence}, the choice of $v_0$ determines whether the gradient flow converges to a global or local minimum. That is, for all $\varepsilon>0$ if $v_0$ lies within the basin of attraction of a global minimizer $J_{\varepsilon}$, then $u_{\varepsilon}$ converges strongly to a minimizer of $J$. However, in practice we can only assess convergence of $u_{\varepsilon}(x)$ and thus, by Theorem \ref{Thm:LocalConvergence}, we can only ensure convergence to a local minimum of $J$.
\end{rk}

\subsection{Finite Difference Scheme on Parallelogram Domains}
We now restrict our attention to the homogeneous case $Q=1$ with $\lambda_1=0$ and $\lambda_2=1$. The domains we consider are parallelograms $\Omega_{\theta}$ defined in coordinates $(\xi,\eta)\in [0,1]\times [0,1]$ by:
\begin{equation}
\Omega_{\theta}=\{(x,y)\in \mathbb{R}^2: (x,y)=(\xi+\eta\cos(\theta),\eta \sin(\theta)) \}
\end{equation} 
with $N=\{\xi=0\}\bigcup \{\eta =1\}$; see Fig \ref{fig:Parallelogram}. The Dirichlet boundary conditions on $S=\partial \Omega_{\theta}\setminus N$ are given by $u_0=\left.\bar{u}_0^A\right|_{S}$ with $\bar{u}_0^A:\Omega_{\theta}\mapsto \mathbb{R}$ defined by
\begin{equation}
\bar{u}^A_0(x,y)=\begin{cases}
-A & x \leq x_0-\delta \\
\displaystyle{A\sin^3\left(\frac{\pi(x-x_0)}{2\delta}\right)} & |x-x_0|<\delta\\
A & x \geq x_0+\delta
\end{cases},
\end{equation}
 where $A,\delta, x_0\in \mathbb{R}$ are parameters satisfying $A>0$ and $0<\delta <x_0<1+\delta$. The Dirichlet boundary condition is chosen so that there is a width $2\delta$ transition between phases at $x_0$; see Fig. \ref{fig:Parallelogram}. We call the point on the parallelogram defined by $\xi=0$ and $\eta=1$ the Neumann corner. 
 
\begin{figure}[ht]
\begin{center}
\includegraphics[width=.65\textwidth]{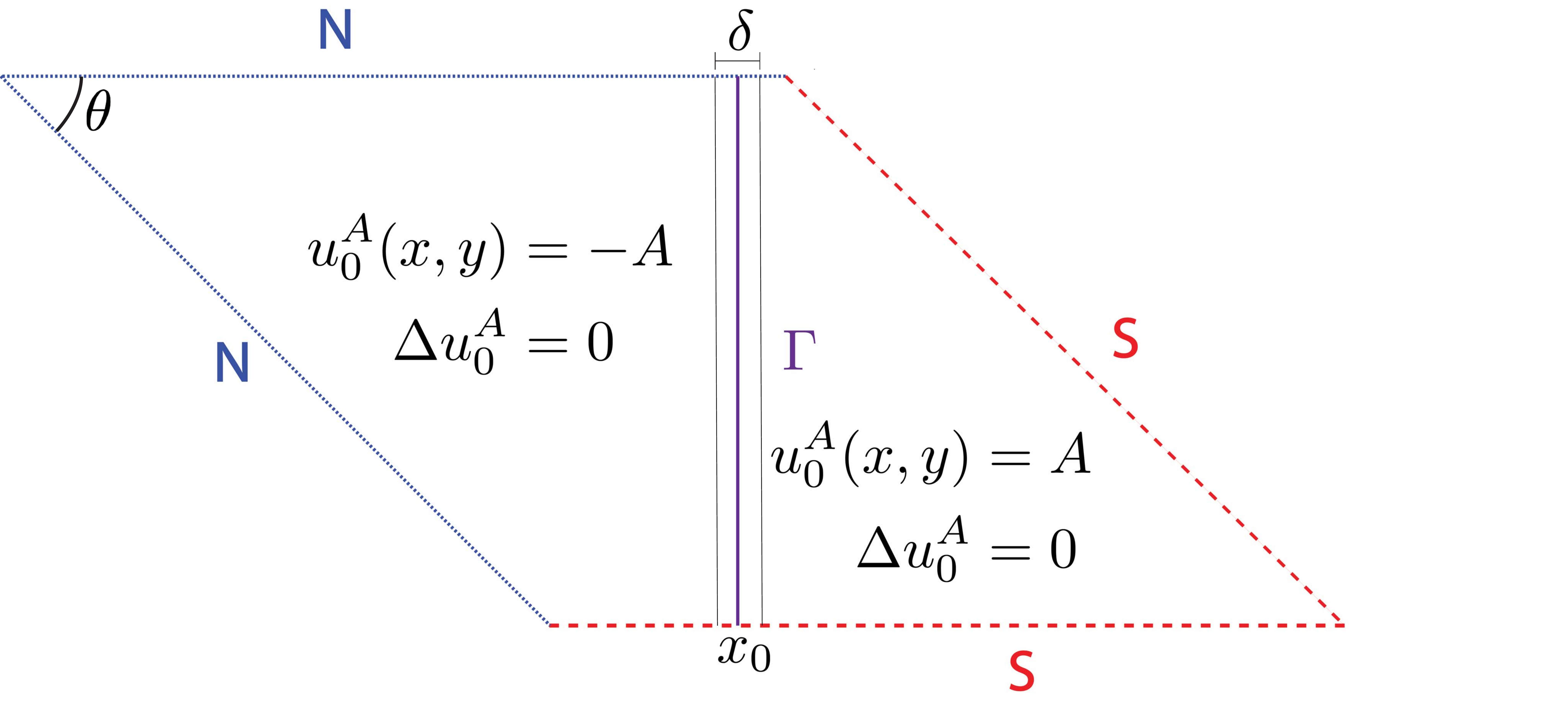}
\end{center}
\caption{Schematic diagram of the parallelogram domains. The function $u_0(x,y)$ defines the Dirichlet boundary conditions on $S$ and is taken as initial data for the gradient flow.}\label{fig:Parallelogram}
\end{figure}

To approximate solutions of the free boundary problem we consider the relaxed functional $J_{\varepsilon}$ with transition layer:
\begin{equation}
\varphi_{\varepsilon}(v)=\begin{cases}
1 & v \geq\varepsilon\\
0 & v \leq 0\\
\frac{1}{2}-\frac{1}{2}\cos\left(\frac{2\pi v}{\varepsilon}\right) & 0<v< \varepsilon.
\end{cases}
\end{equation}
We then apply the gradient flow given by Eq. (\ref{Eqn:GradientFlow}) with initial data:
\begin{equation}
v(0,x)=\bar{u}_0^A(x,y).
\end{equation}
Note that clearly $\bar{u}_0^A(x,y)$ is harmonic outside of the set $|x-x_0|<\delta$ and satisfies the Neumann boundary condition everywhere but does not satisfy Eq. (\ref{Eqn:RelaxedStrongEL}). Indeed, it follows from Eq. (\ref{Eqn:GradientFlow}) and a calculation that if $A>1$, then there exists $0<\delta^{\prime}<\delta$ such that
\begin{equation}
\begin{cases}
\left.v_t\right|_{t=0}\leq 0 &\text{ if } -\delta \leq x< x_0\\
\left.v_t\right|_{t=0}\geq 0 & \text{ if } x_0< x \leq \delta
\end{cases}.
\end{equation}
Consequently, under the gradient flow the positive phase will ``invade'' the negative phase, pushing the free boundary to the left. Moreover, a simple estimate yields the lower bound:
\begin{equation}
J_{\varepsilon}[v_0^A]\geq C\frac{A^2}{\delta}
\end{equation}
and thus $A$ controls the amount of energy in the system. Therefore, $A$ can be used as a knob to control the terminal point of the free boundary under the gradient flow. In particular, as we will show in the next section, there exists a critical value of $A$ in which the free boundary must pass through the Neumann corner or jump from the top to the left Neumann boundary.

To numerically approximate the gradient flow we implement a finite difference scheme. Note that in $(\xi, \eta)$  coordinates the gradient flow for the relaxed problem is given by
\begin{equation}\label{Eqn:FiniteDiff}
\begin{cases}
\displaystyle{v_t=2\left(\csc^2(\theta)v_{\xi \xi}-2\cot(\theta)\csc(\theta)v_{\xi \eta}+\csc^2(\theta)v_{\eta \eta}\right)-\varphi_{\varepsilon}^{\prime}(v)}\\
\displaystyle{\left.-\cot(\theta)v_{\xi}+\csc(\theta)v_{\eta}\right|_{\eta=1}=0}\\
\displaystyle{\left.-\csc(\theta)v_{\xi}+\cot(\theta)v_{\eta}\right|_{\xi =0}=0}\\
\left.v\right|_{S}=u_0(x)\\
v(0,x)=v_0(x)
\end{cases}.
\end{equation}
The spatial derivative operators are approximated using second order centered differences with uniform spacing $h$. On the Neumann boundaries we use ``ghost'' points to close the equations and the evolution in time is implemented using the Crank--Nicolson method \cite{strikwerda2004finite}. The convergence of the gradient flow to a steady state $u^*_h$ is assessed by computing $J_{\varepsilon}$ on each time step. Furthermore, to ensure convergence to a (local) minimizer of the original problem we slave the width of the transition layer to the spacing of the finite difference scheme by setting $\varepsilon=2h$. The mesh is then refined until convergence of $J_{\varepsilon}[u^*_h]$. Therefore, as the mesh is refined, the functions $u^*_h$ form a sequence of approximate minimizers of $J_{2\varepsilon}$, which, by Theorem \ref{Thm:LocalConvergence}, converge to a local minimum of $J$.

\section{Numerical Results}
In this section we present the results of our numerical experiments as well as a discussion of the implication of these results.
\subsection{Obtuse Angle}
In Figure \ref{Fig:Obtuse} we present the results of the finite difference scheme applied to Eq. (\ref{Eqn:FiniteDiff}) for the fixed parameters $x_0=.85$, $\delta=.01$, $\theta=5\pi/4$ and $A=1.19$--$1.33$. Figures \ref{Fig:Obtuse}(A) and \ref{Fig:Obtuse}(B) are contour plots of the numerical approximation to the solution of the free boundary problem. Specifically, Figures \ref{Fig:Obtuse}(A-B) illustrate solutions to the free boundary terminating on the top and left Neumann boundaries respectively. In Figure \ref{Fig:Obtuse}(C) we plot the time evolution of $J_{\varepsilon}$ under the gradient flow for various values of $A$. The numerical scheme indicates that as $A$ is increased the free boundary passes smoothly through the corner point. That is, as $A$ varies there is no discontinuous jump in the spatial coordinate of the terminal point of the free boundary lying on the Neumann boundary. 

\begin{figure}[ht]
\begin{center}
        \includegraphics[width=.75\textwidth]{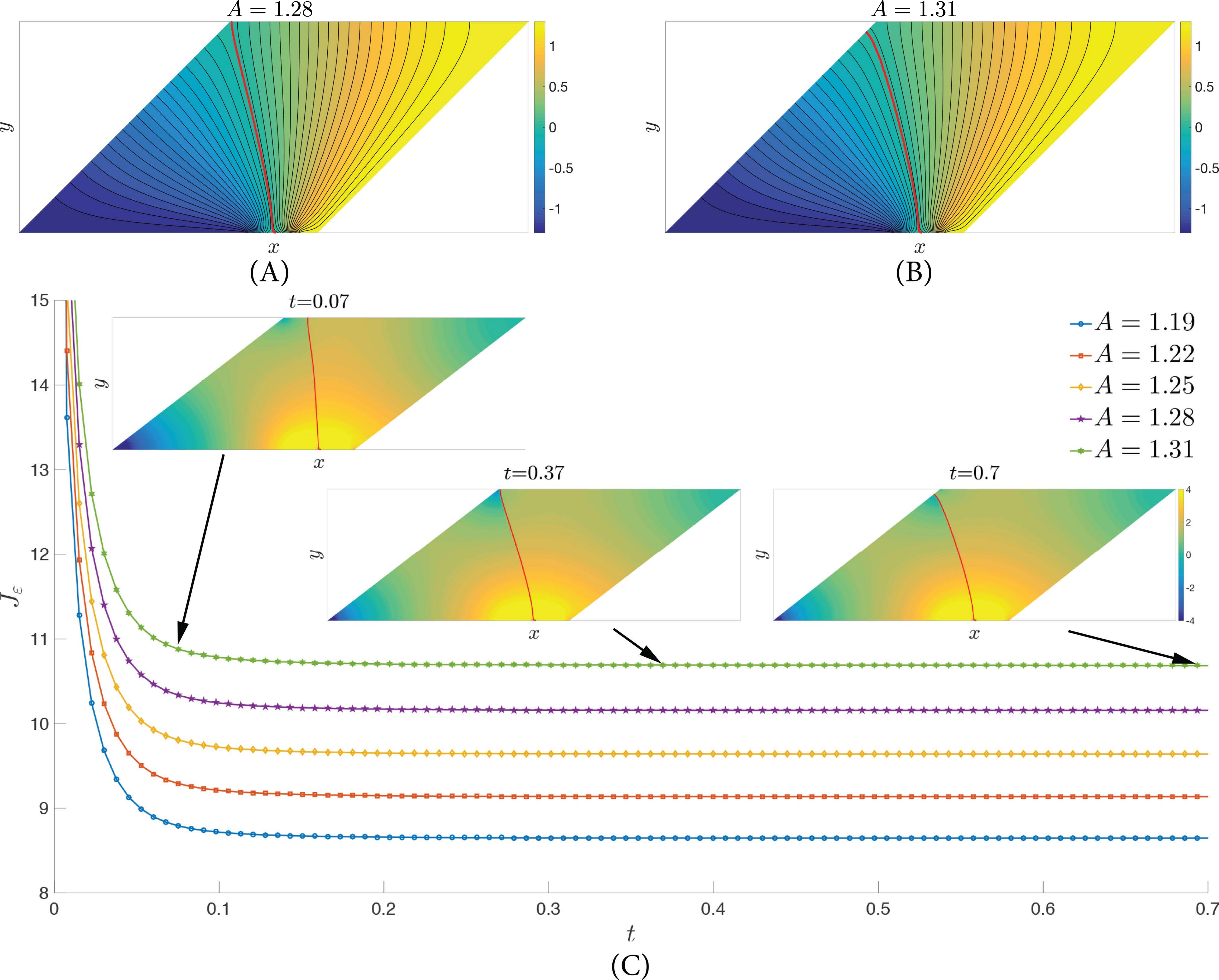}
        \caption{\textbf{(A-B)} Contour plot of numerical solutions to the free boundary problem for $A=1.28$ and $A=1.31$. The solid red curve corresponds to the numerical approximation of the free boundary. \textbf{(C)} Plot of the time evolution of the relaxed energy $J_{\varepsilon}$ under the gradient flow for various values of $A$. The inset figures are time snapshots of the evolution of the free boundary under the gradient flow overlaid on top of a contour plot of $\ln(J_{\varepsilon})$.}
        \label{Fig:Obtuse}
        \end{center}
\end{figure}

\subsection{Right Angle} In Figure \ref{Fig:Right} we present the results of the finite difference scheme applied to Eq. (\ref{Eqn:FiniteDiff}) for the fixed parameters $x_0=.2$, $\delta=.01$, $\theta=\pi/2$ and $A=2.6$--$3.4$. Again, Figures \ref{Fig:Right}(A-B) illustrate solutions to the free boundary terminating on the top and left Neumann boundaries while Figure \ref{Fig:Right}(C) is a plot of the time evolution of $J_{\varepsilon}$. In contrast with the obtuse angle case, as $A$ is increased the free boundary does not pass smoothly through the corner point. However, for this particular geometry this may be an artifact of the numerical scheme. In particular, for all mesh sizes we numerically observed that near the corner point the free boundary enters a ball of radius on the order of the mesh size before ``jumping'' to the other Neumann boundary; see Figure \ref{Fig:Right}(C) insets. That is, the jumping was numerically observed to depend on the mesh size $h$.

\begin{figure}[ht]
\begin{center}
        \includegraphics[width=.75\textwidth]{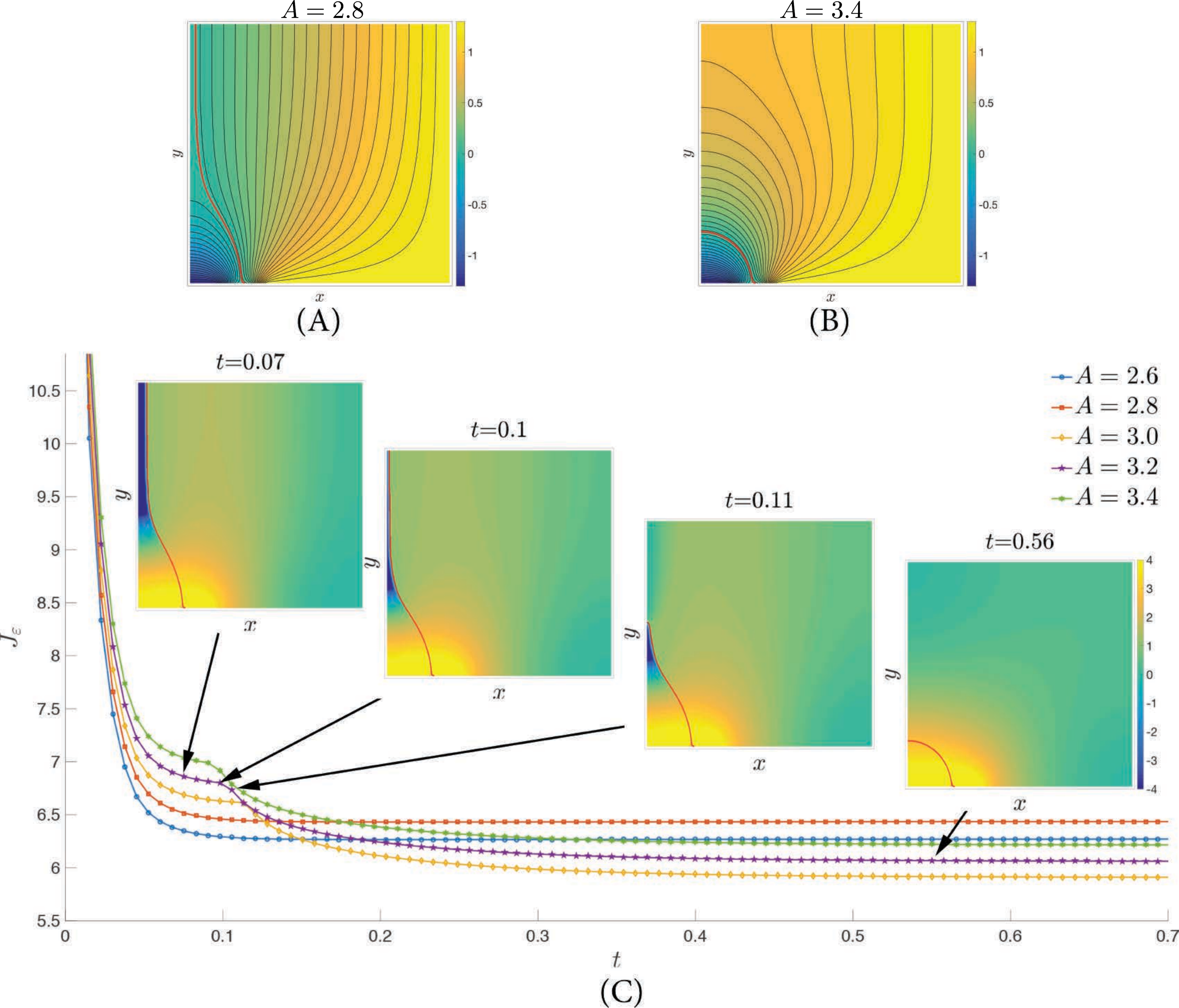}
        \caption{\textbf{(A-B)} Contour plot of numerical solutions to the free boundary problem for $2.8$ and $A=3.4$. The solid red curve corresponds to the numerical approximation of the free boundary. \textbf{(C)} Plot of the time evolution of the relaxed energy $J_{\varepsilon}$ under the gradient flow for various values of $A$. The inset figures are time snapshots of the evolution of the free boundary under the gradient flow overlaid on top of a contour plot of $\ln(J_{\varepsilon})$.}
        \label{Fig:Right}
        \end{center}
\end{figure}

\subsection{Acute Angle} In Figure \ref{Fig:Acute} we present the results of the finite difference scheme applied to Eq. (\ref{Eqn:FiniteDiff}) for the fixed parameters $x_0=.2$, $\delta=.01$, $\theta=\pi/4$ and $A=3.01$--$3.04$. Again, Figures \ref{Fig:Right}(A-B) illustrate solutions to the free boundary terminating on the top and left Neumann boundaries while Figure \ref{Fig:Right}(C) is a plot of the time evolution of $J_{\varepsilon}$. In contrast with both the obtuse and right angle cases, our numerical experiments indicate that the free boundary does not pass smoothly through the corner point and this is not an artifact of the numerical scheme. That is, for sufficiently small $h$ the jumping was numerically observed to not depend on the mesh size. In fact, in contrast with the obtuse and right angles cases as $A$ is varied the steady state of the free boundary never passes through the corner point. More precisely, there exists a neighborhood about the corner point in which the steady state of the free boundary does not enter. 

\begin{figure}[ht]
\begin{center}
        \includegraphics[width=.75\textwidth]{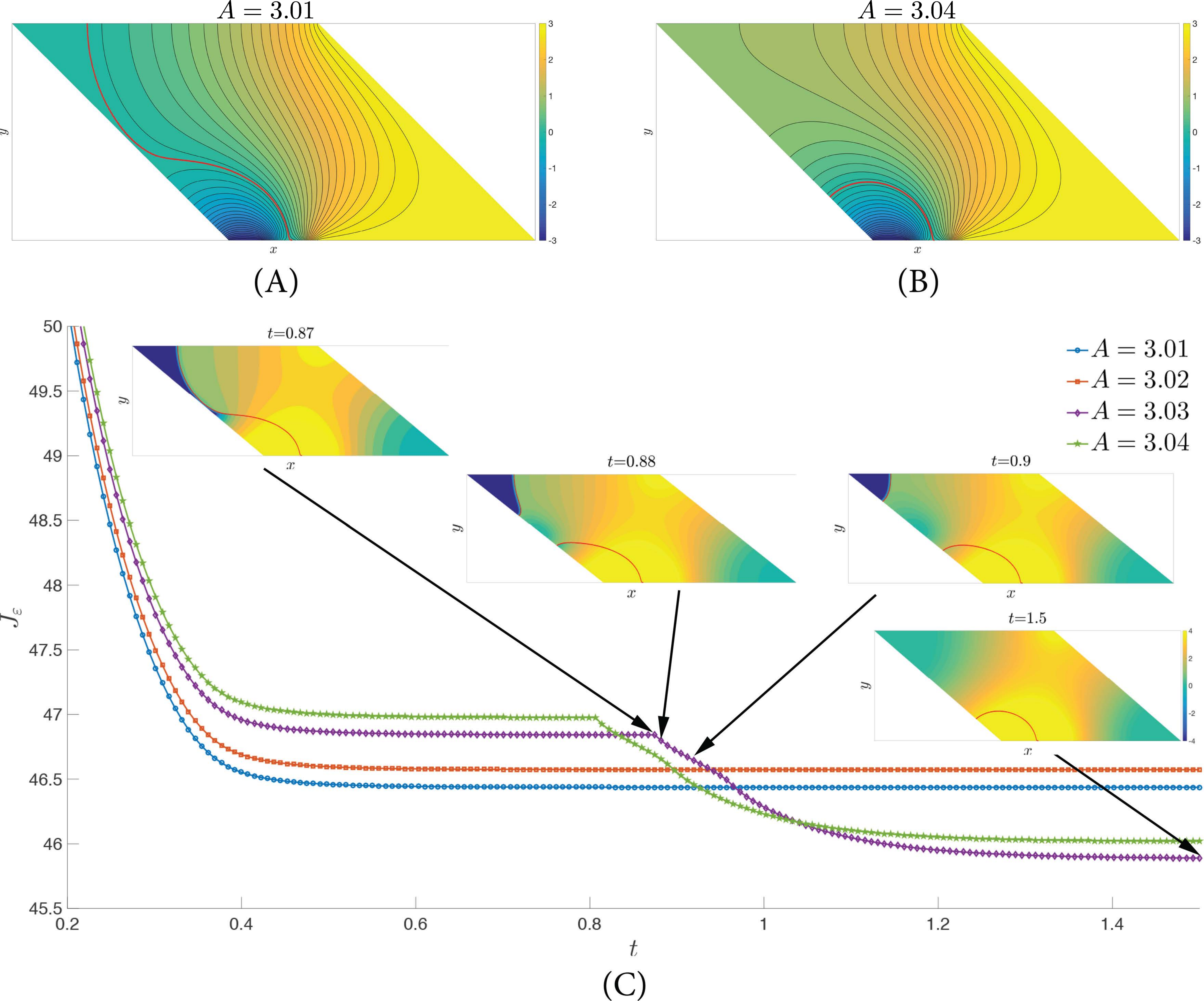}
        \caption{\textbf{(A-B)} Contour plot of numerical solutions to the free boundary problem for $A=3.01$ and $A=3.04$. The solid red curve corresponds to the numerical approximation of the free boundary. \textbf{(C)} Plot of the time evolution of the relaxed energy $J_{\varepsilon}$ under the gradient flow for various values of $A$. The inset figures are time snapshots of the evolution of the free boundary under the gradient flow overlaid on top of a contour plot of $\ln(J_{\varepsilon})$.}
        \label{Fig:Acute}
        \end{center}
\end{figure}

Interestingly, when the free boundary intersects the left Neumann boundary during the gradient flow it splits into two separate curves with a transient ``zero'' phase enclosing the corner point. This zero phase is then rapidly invaded by the positive phase and disappears. During this transient period the energy is rapidly decreasing before reaching a steady state in which the free boundary enclosed the lower left corner of the parallelogram. However, before passing through the Neumann corner the change in the energy is very slow. When viewed as a dynamical system, this type of transition is reminiscent of a saddle node bifurcation in which as $A$ is increased a stable steady state disappears and the system is driven to a separate equilibrium. In particular, the slowing down of the dynamics is likely the result of the ``ghost'' of the previous stable equilibrium. If this is the case, then before the bifurcation there are necessarily at least two steady state solutions to the gradient flow.

\section{Conclusions and Future Work}

In this paper, we have shown the Lipschitz continuity of solutions to a two-phase free boundary problem near a convex Neumann fixed boundary in two dimensions.  We have tested and numerically validated the hypothesis that the free and fixed boundaries should intersect orthogonally in this context.  A major direction of future work will be to validate this analytically. Another important future direction is to generalize to higher dimensions.    

Our numerical experiments indicate that this orthogonality generates interesting behavior near a right or acute angle in the boundary.  In particular, the free boundary, as it approaches a right angle, becomes tangent to the other piece of the angle and, when ultimately it "flips" past the angle, travels a significant distance very quickly.  On the other hand, as the free boundary approaches an acute angle, it cannot approach particularly close and there is a "forbidden" region" due to the need for orthogonality and the energy constraints.  Several interesting questions arise from these observations.  Is the asymmetry in the way the free boundary jumps across a corner an artifact of our numerical scheme or an indication of the presence of multiple local minimizers? Is the zero phase that is generated in the "forbidden" region a numerical artifact?  What general theory can we develop for more general domains than simply parallelograms?  What would a three-dimensional version of these numerics show?  

\bibliographystyle{abbrv} 
\bibliography{ref}

\end{document}